\documentclass[a4paper]{article}
\usepackage{amsthm}
\usepackage{amsmath}
\usepackage{amssymb}
\usepackage{amsfonts}
\usepackage{graphicx}
\usepackage[utf8]{inputenc}
\usepackage[T1]{fontenc}
\usepackage[margin=2cm]{geometry}

\newcommand{\uu}[1]{2\uparrow \uparrow #1}
\newtheorem{thm}{Theorem}
\newtheorem{lem}{Lemma}
\newtheorem{cor}{Corollary}
\theoremstyle{definition}
\newtheorem{df}{Definition}

\title{Further improving of upper bound on a geometric Ramsey problem}
\author{Eryk Lipka \\ Zaremba Society of Mathematicians – Students of the Jagiellonian University \\ Institute of Mathematics of the Pedagogical University of Cracow \\
\texttt{eryklipka0@gmail.com}}
\date{May 2019}

\begin{document}
\maketitle
\begin{abstract}
    We consider following geometric Ramsey problem: find the least dimension $n$ such that for any 2-coloring of edges of complete graph on the points $\{\pm 1\}^n$ there exists 4-vertex coplanar monochromatic clique. Problem was first analyzed by Graham and Rothschild \cite{GR} and they gave an upper bound: $n\le F(F(F(F(F(F(F(12)))))))$, where $F(m) = 2\uparrow^m3$. In 2014 Lavrov, Lee and Mackey \cite{ML} greatly improved this result by giving upper bound $n< 2\uparrow\uparrow\uparrow 6 < F(5)$. In this paper we revisit their estimates and reduce upper bound to $n< 2\uparrow\uparrow\uparrow 5$.
    
\end{abstract}

\section{Setting}

\begin{df}
Given $n,c,d\in \mathbb Z_+ $ let \textbf{Hales-Jewett number} $\text{HJ}\left(n,c,d\right)$ be the least integer $k$ with the following property.
For any $c$-coloring $D$ of $\left\{0,1,\ldots, n-1\right\}^k$ there exists an injective function $\rho:\left\{0,1,\ldots, n-1\right\}^d \rightarrow\left\{0,1,\ldots, n-1\right\}^k$ such, that
\begin{equation*}
\forall_{1\le i\le k}\left[\exists_{1\le j \le d}\,\rho_i\left(y_1,\ldots,y_d\right) = y_d  \vee \exists_{0\le j\le n-1}\,\rho_i\left(y_1,\ldots,y_d\right) = j\right],
\end{equation*}
and $\rho\left(\left\{0,1,\ldots, n-1\right\}^d\right)$ is $D$-monochromatic.
\end{df}

\begin{df}
Given $n,c,d\in \mathbb Z_+ $ let \textbf{Tic-Tac-Toe number} $\text{TTT}\left(n,c,d\right)$ be the least integer $k$ with the following property.
For any $c$-coloring $D$ of $\left\{0,1,\ldots, n-1\right\}^k$ there exists an injective function $\rho:\left\{0,1,\ldots, n-1\right\}^d \rightarrow\left\{0,1,\ldots, n-1\right\}^k$ such, that
\begin{equation} \label{eq:propttt}
\forall_{1\le i\le k}\left[\exists_{1\le j \le d}\,\rho_i\left(y_1,\ldots,y_d\right) = y_d \vee\exists_{1\le j \le d}\,\rho_i\left(y_1,\ldots,y_d\right) = n-1-y_d  \vee \exists_{0\le j\le n-1}\,\rho_i\left(y_1,\ldots,y_d\right) = j\right],
\end{equation}
and $\rho\left(\left\{0,1,\ldots, n-1\right\}^d\right)$ is $D$-monochromatic. Image of such a function is called a \textbf{$d$-dimensional Tic-Tac-Toe Subspace}.
\end{df}

\begin{df}
Given $d \in \mathbb Z_+$ let $\text{Graham}\left(d\right)$ be the smallest dimension $k$ such that for every edge-coloring of a complete graph on the points $\{\pm 1\}^k$ there exists an injective function $\rho:\left\{\pm 1\right\}^d \rightarrow\left\{\pm 1\right\}^k$ with 
$$\forall_{1\le i\le k}\left[\exists_{1\le j \le d}\,\rho_i\left(y_1,\ldots,y_d\right) = y_d \vee\exists_{1\le j \le d}\,\rho_i\left(y_1,\ldots,y_d\right) = -y_d  \vee \exists_{0\le j\le n-1}\,\rho_i\left(y_1,\ldots,y_d\right) = j\right],$$
and all edges between the points of $\rho\left(\left\{\pm 1\right\}^d\right)$ have the same color.
\end{df}
In particular, $\text{Graham}\left(2\right)$ is the smallest integer $k$, such that for every edge-coloring of a complete graph on the points $\{\pm 1\}^k$ there exist four coplanar vertices such that all six edges between them are monochromatic. Our goal is to give a better upper bound for that value.
It has been proven in \cite{ML}, that $\text{Graham}\left(2\right)\le \text{TTT}\left(4, 2, 6\right) +1$, and then, using obvious inequality $\text{HJ}\left(n,c,d\right) \ge \text{TTT}\left(n,c,d\right)$ it was shown that $\text{Graham}\left(2\right) < 2\uparrow\uparrow\uparrow 6$. Our approach is to not use the Hales-Jewett function, because TTT$\left(\cdot, c, d\right)$ and HJ$\left(\cdot, c, d\right)$ have similar growth rate, but initial values of TTT$\left(\cdot, c, d\right)$ are much smaller.

\begin{df}
Given $n,c,l\in \mathbb Z_+ $ let $\text{Cub}\left(n,c,l\right)$ be the least integer $k$ with the following property.
For any $c$-coloring $D$ of $X=\left\{0,1,\ldots, n-1\right\}^k$ there exists $c$-coloring $D'$ of $Y=\left\{0,1,\ldots, n-1\right\}^l$ and an injective function $\pi : Y \rightarrow X$ such, that
\begin{equation} \label{eq:property1}
\forall_{1\le i\le k}\left[\exists_{1\le j \le l}\pi_i\left(y_1,\ldots,y_l\right) = y_j  \vee \exists_{0\le j\le n-1}\pi_i\left(y_1,\ldots,y_l\right) = j\right],    
\end{equation}
\begin{equation} \label{eq:property2}
\forall_{y\in Y} D\left(\pi\left(y\right)\right) = D'\left(y\right),    
\end{equation}
\begin{equation} \label{eq:property3}
\forall_{1\le i\le l} D'\left(y_1,y_2,\ldots,y_{i-1}, n-1, y_{i+1},\ldots, y_l\right) = D'\left(y_1,y_2,\ldots,y_{i-1}, n-2, y_{i+1},\ldots, y_l\right).    
\end{equation}
In other words, values $n-2$ and $n-1$ are not distinguished by induced coloring of $Y$. 
\end{df}

\begin{lem}
Let $n,c,l \in \mathbb Z_+$, then $\text{Cub}\left(n,c,l\right) \le l \cdot f\left(l,c^{n^l}\right)$, where
$$f\left(l,k\right) =
\begin{cases}
k^{f\left(l-1,k\right)^{2l-2}}+1 & \text{for } l>1\\
k+1 & \text{for } l = 1
\end{cases}.$$
\end{lem}
\begin{proof}
This is straightforward conclusion from chapter 1 of \cite{SH}. This fact was used to show, that $\text{HJ}\left(n+1, c, d\right)\le \text{Cub}\left(n+1,c,\text{HJ}\left(n,c,d\right)\right)\le \text{HJ}\left(n,c,d\right) \cdot f\left(\text{HJ}\left(n,c,d\right), c^{{n+1}^{\text{HJ}\left(n,c,d\right)}}\right)$.
\end{proof}

\begin{lem}
Let $k,l\in \mathbb Z_+$ and $f$ be defined as above, then $2l < k \Rightarrow f\left(l,k\right) < k \uparrow \uparrow 2l.$
\end{lem}
\begin{proof}
For $l=1$ it is obviously true as $k+1 < k^k=k \uparrow \uparrow 2$ for $k>2l=2$. By induction it is true for any $l$ because
$$f\left(l,k\right) < k^{{\left(k\uparrow \uparrow 2l-2\right)}^{2l-2}} < k^{{\left(k\uparrow \uparrow 2l-2\right)}^k} < k\uparrow\uparrow 2l.$$
\end{proof}

\begin{lem}
Let $c,d\in \mathbb Z_+$, then $\text{TTT}\left(2, c, d\right) \le \frac{c}{2}\cdot 3^{d}$.
\end{lem}
\begin{proof}
First, we notice $\text{TTT}\left(2, c, 1\right) = \left\lceil \log_2\left(c+1\right)\right\rceil\le c$ as a line connecting any two points in $\left\{0,1\right\}^k$ has property (\ref{eq:propttt}), so we just need to have more points than colors. Define $r_1=\left\lceil \log_2\left(c+1\right)\right\rceil, r_i = \left\lceil \log_2\left(c\cdot \prod_{j<i}  \genfrac{(}{)}{0pt}{}{2^{r_j}}{2}+1\right)\right\rceil$, then by pigeonhole principle $\text{TTT}\left(2, c, d\right) \le \sum_{j\le d} r_j$. Because $r_i \le 3\cdot r_{i-1}$ then $\sum_{j\le d} r_j \le r_1 \cdot \frac{3^d-1}{2}$.
\end{proof}

\begin{cor}
By carefully repeating previous proof we can get even better estimate for certain values, in particular for $c=2$ we have $\left(r_i\right) = \left(2,4,11,32,95,284,\ldots\right)$ so $\text{TTT}\left(2, 2, 6\right) \le 428$.
\end{cor}

\section{Main Result}

\begin{lem} For $n\ge 2$
$$\text{TTT}\left(n+2, c, d\right)\le \text{Cub}\left(n+2,c,\text{Cub}\left(n+1,c,\text{TTT}\left(n, c, d\right)\right)\right).$$
\end{lem}
\begin{proof}
We will basicaly repeat proof of Lemma 1.4 from \cite{SH}, but with TTT instead of Hales-Jewett function. Define 
\begin{align*}
    X&=\left\{0,1,\ldots, n+1\right\}^{\text{Cub}\left(n+2,c,\text{Cub}\left(n+1,c,\text{TTT}\left(n, c, d\right)\right)\right)},\\
    Y&=\left\{0,1,\ldots, n+1\right\}^{\text{Cub}\left(n+1,c,\text{TTT}\left(n, c, d\right)\right)},\\
    Y'&=\left\{0,1\ldots,n\right\}^{\text{Cub}\left(n+1,c,\text{TTT}\left(n, c, d\right)\right)},\\
    Z&=\left\{0,1,\ldots, n\right\}^{\text{TTT}\left(n, c, d\right)},\\
    Z'&=\left\{0,1,\ldots, n-1\right\}^{\text{TTT}\left(n, c, d\right)}.
\end{align*}

Let $D$ be any $c$-coloring of $X$, by definition of Cub we have induced $c$-coloring $D'$ of $Y$ and embedding $\pi: Y \rightarrow X$ such that properties (\ref{eq:property1}),(\ref{eq:property2}),(\ref{eq:property3}) hold (for $n := n+2$). Again, by definition of Cub we have induced $c$-coloring $D''$ of $Z$ and embedding $\tau: Z \rightarrow Y'$ such that properties (\ref{eq:property1}),(\ref{eq:property2}),(\ref{eq:property3}) hold (for $n:=n+1$). By definition of TTT there exists an injective function $\rho: \left\{0,1,\ldots, n-1\right\}^d \rightarrow Z'$ with property (\ref{eq:propttt}) and its image is $D''$-monochromatic.
Let $\sigma : Y' \rightarrow Y'$ be defined as $\sigma\left(y_1,y_2,\ldots\right) = \left(n-y_1,n-y_2,\ldots\right)$, also let $\gamma: Y'\rightarrow Y, \zeta: Z' \rightarrow Z$ be natural embeddings.

Define $\rho' : \left\{0,1,\ldots, n-1\right\}^d \rightarrow X$ as 
$$\rho'\left(x_1,x_2,\ldots,x_d \right)=\pi\circ\gamma\circ\sigma\circ\tau\circ\zeta\circ\rho\left(x_1,x_2,\ldots,x_d \right).$$
\includegraphics[width=\textwidth]{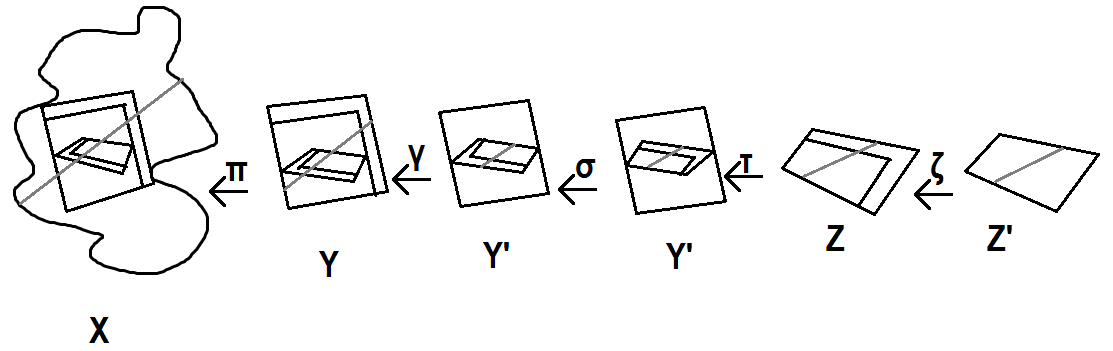}
It is easy to check, that $\rho'\left(\left\{0,1,\ldots, n-1\right\}^d\right) \subset \left\{1,2,\ldots, n\right\}^{\dim X}$ is $D$-monochromatic and
$$\forall_{1\le i\le \dim{X}}\left[\exists_{1\le j \le d}\rho'_i\left(x_1,\ldots,x_d\right) = x_j+1 \vee\exists_{1\le j \le d}\rho'_i\left(x_1,\ldots,x_d\right) = n-x_j  \vee \exists_{1\le j\le n}\rho'_i\left(x_1,\ldots,x_d\right) = j\right].$$
Now, we define function $\rho'':\left\{0,1,\ldots, n,n+1\right\}^d \rightarrow X$ in a following way
$$\rho''_i\left(x_1,\ldots,x_d\right) =
\begin{cases}
x_j & \text{if } \exists_{1\le j \le d}\rho'_i\left(x_1,\ldots,x_d\right) = x_j+1\\
n+1-x_j & \text{if } \exists_{1\le j \le d}\rho'_i\left(x_1,\ldots,x_d\right) = n-x_j\\
j & \text{if } \exists_{1\le j\le n}\rho'_i\left(x_1,\ldots,x_d\right) = j
\end{cases}.$$
This function satisfies $\rho'\left(x_1,\ldots,x_d\right) = \rho''\left(x_1+1,\ldots,x_d+1\right)$, so $\rho'\left(\left\{0,1,\ldots, n-1\right\}^d\right) = \rho''\left(\left\{1,2\ldots, n\right\}^d\right)$, and image of $\rho''$ is a Tic-Tac-Toe subspace. Because $\pi$ and $\tau$ have property (\ref{eq:property3}) this image is also $D$-monochromatic, so $\text{TTT}\left(n+2,c,d\right) \le \dim X.$
\end{proof}

\begin{lem}
$\text{Cub}\left(3,2,\text{TTT}\left(2, 2, 6\right)\right) < \uu{5137} $.
\end{lem}
\begin{proof}
\begin{align*}
    \text{Cub}\left(3,2,\text{TTT}\left(2, 2, 6\right)\right) &\le \text{TTT}\left(2, 2, 6\right)\cdot f\left(\text{TTT}\left(2, 2, 6\right),2^{3^{\text{TTT}\left(2, 2, 6\right)}}\right)\\
    &\le 428\cdot f\left(428,2^{3^{428}}\right)
    \le 428\cdot\left(2^{3^{428}} \uparrow \uparrow 856\right)\\
    &< 2^9 \cdot \left(\left(\uu{6}\right) \uparrow \uparrow 856\right)
    < \uu{\left(6\cdot 856 +1\right)}
    = \uu{5137}. 
\end{align*}
\end{proof}

\begin{lem}
$\text{TTT}\left(4, 2, 6\right) + 1 < \uu{\left(\uu{5138}\right)}$.
\end{lem}
\begin{proof}
\begin{align*}
\text{TTT}\left(4, 2, 6\right)+1 &\le 1+\text{Cub}\left(4, 2, \text{Cub}\left(3, 2, \text{TTT}\left(2, 2, 6\right)\right)\right)\\
&\le 1+\text{Cub}\left(4, 2, \uu{5137} \right)\\
&\le 1+\left(\uu{5137}\right) \cdot f\left(\uu{5137},2^{4^{\uu{5137}}}\right) \\
&\le 1+\left(\uu{5137}\right) \cdot \left( \left(2^{4^{\uu{5137}}}\right) \uparrow \uparrow \left(2\cdot \uu{5137}\right) \right)\\
&< \left(\uu{5138}\right) \cdot \left( \left(\uu{5140}\right) \uparrow \uparrow \left(2\cdot \uu{5137}\right) \right)\\
&< \left(\uu{5138}\right) \cdot \left( \uu{\left(5140\cdot 2\cdot \uu{5137}\right)} \right)\\
&< \uu{\left(10280 \cdot \uu{5137} +1\right)}\\
&< \uu{\left(\uu{5138}\right)}.
\end{align*}
\end{proof}

\begin{thm}
$\text{Graham}\left(2\right) < 2\uparrow\uparrow\uparrow 5$.
\end{thm}
\begin{proof}
From \cite{ML} we know, that $\text{Graham}\left(2\right)\le \text{TTT}\left(4, 2, 6\right) +1$, so $\text{Graham}\left(2\right) < \uu{\left(\uu{5138}\right)} < 2\uparrow\uparrow\uparrow 5$.
\end{proof}

\end{document}